\documentclass[letter,10pt]{amsart}
\usepackage{amsmath, amsfonts, amssymb,amscd,graphics,graphicx,eucal,setspace} 
\usepackage{latexsym} 
\usepackage[dvipsnames]{xcolor}

\usepackage{url}
\usepackage{tikz,subcaption,mathtools,tikz-cd}
\usepackage{booktabs}
\usetikzlibrary{decorations.markings}

\numberwithin{equation}{section}

\theoremstyle{plain}
\newtheorem{theorem}{Theorem}[section]
\newtheorem{lemma}[theorem]{Lemma}
\newtheorem{corollary}[theorem]{Corollary}
\newtheorem{proposition}[theorem]{Proposition}

\newtheorem{thmx}{Theorem}

\theoremstyle{definition}
\newtheorem{remark}[theorem]{Remark}
\newtheorem{example}[theorem]{Example}
\newtheorem{definition}[theorem]{Definition}

\newtheorem*{question}{Question}

\def\C{\mathbb C}
\def\R{\mathbb R}

\def\Z{\mathbb Z}

\def\A{\mathcal M}

\DeclareMathOperator{\SL}{SL}

\DeclareMathOperator{\Lie}{Lie}

\DeclareMathOperator{\Int}{Int}
\DeclareMathOperator{\Hom}{Hom}

\DeclareMathOperator{\Ad}{Ad}

\def\aret{\mathcal{R}^a}
\def\gret{\mathcal{R}^g}
\def\mret{\mathcal{R}^m}

\def\y{q}

\usepackage{caption}
\usepackage[labelformat=simple]{subcaption}

\begin{document}
\title[Torus orbit closures and retractions]{Torus orbit closures in flag varieties and\\ retractions on  Weyl groups}

\date{\today}

\author[Eunjeong Lee]{Eunjeong Lee}
\address[E. Lee]{Center for Geometry and Physics, Institute for Basic Science (IBS), Pohang 37673, Republic of Korea}
\email{eunjeong.lee@ibs.re.kr}

\author[Mikiya Masuda]{Mikiya Masuda}
\address[M. Masuda]{Osaka City University Advanced Mathematics Institute (OCAMI) \& Department of Mathematics, Graduate School of Science, Osaka City University, Sumiyoshi-ku, Sugimoto, 558-8585, Osaka, Japan}
\email{masuda@osaka-cu.ac.jp}

\author[Seonjeong Park]{Seonjeong Park}
\address[S. Park]{Department of Mathematical Sciences, Korea Advanced Institute for Science and Technology (KAIST), Daejeon 34141, Republic of Korea}
\email{psjeong@kaist.ac.kr}



\begin{abstract}
A finite Coxeter group $W$ has a natural metric $d$ and if $\A$ is a subset of $W$, then for each $u\in W$, there is $\y\in \A$ such that $d(u,\y)=d(u,\A)$. Such $\y$ is not unique in general but if $\A$ is a Coxeter matroid, then it is unique, and we define a retraction $\mret_\A\colon W\to \A\subset W$ so that $\mret_\A(u)=\y$. 

The $T$-fixed point set $Y^T$ of a $T$-orbit closure $Y$ in a flag variety $G/B$ is a Coxeter matroid, where $G$ is a semisimple algebraic group, $B$ is a Borel subgroup, and $T$ is a maximal torus of $G$ contained in $B$. We define a retraction $\gret_{Y}\colon W\to Y^T\subset W$ geometrically, where $W$ is the Weyl group of $G$, and show that $\gret_{Y}=\mret_{Y^T}$. We introduce another retraction $\aret_\A\colon W\to \A\subset W$ algebraically for an arbitrary subset $\A$ of $W$ when $W$ is a Weyl group of classical Lie type, and show that $\aret_\A=\mret_\A$ when $\A$ is a Coxeter matroid. 
\end{abstract}

\maketitle

\setcounter{tocdepth}{1}
\tableofcontents

\section{Introduction}

Let $W$ be a finite Coxeter group. We consider the metric $d$ on $W$ defined by 
\[
d(v,w):=\ell(v^{-1}w)=\ell(w^{-1}v)
\]
where $\ell(\ )$ denotes the length function on $W$. For a subset $\A$ of $W$, we define 
\begin{equation*} 
d(v,\A)\coloneqq \min\{d(v,w)\mid w\in \A\}.
\end{equation*}
For each $u\in W$, there exists an element $\y\in \A$ such that $d(u,\y)=d(u,\A)$ but such $\y$ is not unique in general. However, it is unique when $\A$ is a Coxeter matroid. 

A Coxeter matroid we treat in this paper is a subset $\A$ of $W$ satisfying the \emph{Maximality Property}, that is, for any $u\in W$, there is a unique element $v\in \A$ such that $w\le^u v$ for all $w\in \A$, where $w\le^u v$ means $u^{-1}w\le u^{-1}v$ in the Bruhat order. Since the multiplication by the longest element of $W$ reverses the Bruhat order on~$W$, the Maximality Property is equivalent to the \emph{Minimality Property}, that is, for each $u\in W$, there exists a unique minimal element in $\A$ with respect to $\leq^{u}$. We denote the unique minimal element by $\mret_\A(u)$, so 
\begin{equation} \label{eq:rcoxeter}
\text{$\mret_\A(u)\le^u w$\quad for all $w\in \A$}.
\end{equation}
Since $\mret_\A(u)=u$ for $u\in \A$, the map 
$$\mret_\A\colon W\to \A\subset W$$ is a retraction of $W$ onto $\A$, which we call a \emph{matroid retraction}. One can easily see the following. 

\begin{proposition}[Proposition~\ref{prop_mret_minimal_distance}] \label{prop:intro}
If $\A$ is a Coxeter matroid of $W$, then 
\begin{enumerate}
\item for each $u\in W$, there is a \emph{unique} $\y\in \A$ such that $d(u,\y)=d(u,\A)$, and \item $\y=\mret_\A(u)$.
\end{enumerate} 
\end{proposition}

It is known that a torus fixed point set in a flag variety is a Coxeter matroid. 
Let $G$ be a semisimple algebraic group over $\C$, $B$ a Borel subgroup of $G$, and $T$ a maximal torus of $G$ contained in $B$. Then the left multiplication by the torus $T$ on $G$ induces the action of $T$ on $G/B$ and we naturally identify the $T$-fixed point set $(G/B)^T$ in $G/B$ with the Weyl group $W$ of $G$. 

For $u\in W$, we consider a cone in the real Lie algebra $\mathfrak{t}_\R$ of $T$ defined by 
\[
C(u)\coloneqq\{\lambda\in\mathfrak{t}_\R\mid\langle u(\alpha),\lambda\rangle\leq 0 \quad \text{ for all simple roots }\alpha\}.
\]
The interior of $C(u)$ is the Weyl chamber corresponding to $u$. 
We regard $\mathfrak{t}_\R$ as $\mathfrak{t}_\Z\otimes\R$, where $\mathfrak{t}_\Z:=\Hom(\C^*,T)$ is the group of algebraic homomorphisms from $\C^*=\C\backslash\{0\}$ to $T$. 

Choose a point $x$ of $G/B$ and let $Y$ be the closure of the $T$-orbit of $x$ in $G/B$. Then, for any $\lambda_u\in \Int (C(u))\cap \mathfrak{t}_\Z$, we have 
\begin{equation*} 
\lim_{t\to 0}\lambda_u(t)\cdot x\in Y^T\subset (G/B)^T=W.
\end{equation*}
This limit point does not depend on the choice of $x$ defining $Y$ and $\lambda_u$. So, we denote the limit point by $\gret_Y(u)$. One can see that 
\[
\gret_Y\colon W\to Y^T\subset W
\]
is a retraction of $W$ onto $Y^T$ and we call it a \emph{geometric retraction}. Remember that $Y^T$ is a Coxeter matroid of $W$, so we also have the matroid retraction $\mret_{Y^T}\colon W\to Y^T$. Under this situation, our first main result is the following. 

\begin{thmx}[{Theorem~\ref{theo:torus-coxeter}}] \label{theo:A}
$\gret_Y=\mret_{Y^T}$ for any $T$-orbit closure $Y$ in $G/B$. 
\end{thmx}

\begin{remark}
When $Y$ is \textit{generic}, which means that $Y^T=(G/B)^T$, $Y$ is known to be a smooth toric variety (called a \textit{permutohedral variety} when $G$ is of type~$A$). In this case, the geometric retraction $\gret_Y$ is the idenitity map and the maximal cones in the fan of $Y$ are the cones $C(u)$ $(u\in W)$. Unless $Y$ is generic, $Y$ can be singular but the Orbit-Cone Correspondence in the theory of toric varieties tells us that the maximal cone corresponding to $y\in Y^T$ in the fan of $Y$ is given by $\bigcup_{u\in (\gret_Y)^{-1}(y)}C(u)$, see Remark~\ref{rmk_normal} for more details. 
\end{remark}

When a finite Coxeter group $W$ is a product $\prod_{j=1}^mW_j$ of Weyl groups $W_j$'s of classical Lie types, we introduce what we call an \emph{algebraic retraction} 
\[
\aret_\A\colon W\to \A \subset W
\]
algebraically for $\A=\prod_{j=1}^m\A_j$, where $\A_j$ is an \emph{arbitrary} subset of $W_j$ for each~$j$ (see Definition~\ref{defi:algebraic_retraction} and \eqref{equation_aret_product_fixed_points}). 
We note that if $\A$ is a Coxeter matroid of $W=\prod_{j=1}^mW_j$, then $\A$ is a product $\prod_{j=1}^m\A_j$ and each $\A_j$ is a Coxeter matroid of $W_j$. Under this situation, our second main result is the following. 

\begin{thmx}[{Theorem~\ref{thm:aret-mret}}] \label{theo:B}
$\aret_\A=\mret_\A$ for any Coxeter matroid $\A$ of $W$.
\end{thmx}

If a subset $\A$ of a finite Coxeter group $W$ is a Coxeter matroid, then 
\begin{enumerate}
\item[$(*)$] for each $u\in W$, there is a \emph{unique} $\y\in \A$ such that $d(u,\y)=d(u,\A)$
\end{enumerate}
as mentioned in Proposition~\ref{prop:intro}. This is a necessary condition for a subset $\A$ of $W$ to be a Coxeter matroid. However, this necessary condition $(*)$ is not a sufficient condition as is shown in Example~\ref{example_not_unique_closest}, where $\A$ consists of two elements. We note that the algebraic retraction is defined for an \emph{arbitrary} subset $\A$ of $W$ when $W$ is a Weyl group of classical Lie type, and Proposition~\ref{prop:intro} and Theorem~\ref{theo:B} show that if $\A$ is a Coxeter matroid of $W$, then the unique element $\y$ in $(*)$ above must be given by $\aret_\A(u)$. We ask whether these two necessary conditions are sufficient, namely

\begin{question}
Let $W$ be a Weyl group of classical Lie type. Suppose that a subset $\A$ of $W$ satisfies the following two conditions:
\begin{enumerate}
\item for each $u\in W$, there is a unique $\y\in \A$ such that $d(u,\y)=d(u,\A)$, and
\item $q=\aret_\A(u)$.
\end{enumerate}
Then, is $\A$ a Coxeter matroid?
\end{question}

We do not know any counterexample to the question and answer it affirmatively when $W$ is a symmetric group and $\A$ consists of two elements (Proposition~\ref{prop:two-elements}). 

This paper is organized as follows. In Section~\ref{sect:Coxeter_matroids}, we introduce the matroid retraction and prove Proposition~\ref{prop:intro}. We also recall the characterization of Coxeter matroids by Gelfand-Serganova in terms of the associated polytopes. In Section~\ref{sect:torus-orbit}, we define the geometric retraction, reformulate it using a Bruhat decomposition of $G/B$, and prove Theorem~\ref{theo:A}. A Coxeter matroid is called representable if it can be realized as the torus fixed point set of a torus orbit closure in $G/B$. We discuss the representability when $G$ is of type A in Section~\ref{sect:representability}. In Section~\ref{sect:Retraction on the Weyl group}, we define the algebraic retraction and prove Theorem~\ref{theo:B}. In Section~\ref{sec:characterization}, we discuss the above question. 

\medskip

\noindent\textbf{Acknowledgments.}
The authors thank to Satoru Fujishige, Gabriel C. Drummond-Cole, Damien Lejay, and Kyeong-Dong Park for their interest and valuable discussions. They also thank the anonymous referees for their valuable comments to improve the presentation of the paper drastically. 

Lee was supported by IBS-R003-D1. Masuda was supported in part by JSPS Grant-in-Aid for Scientific Research 19K03472 and a bilateral program between JSPS and RFBR. Park was supported by the Basic Science Research Program through the National Research Foundation of Korea (NRF) funded by the Government of Korea (NRF-2018R1A6A3A11047606). This work was partly supported by Osaka City University Advanced Mathematical Institute (MEXT Joint Usage/Research Center on Mathematics and Theoretical Physics JPMXP0619217849).

\section{Coxeter matroids}\label{sect:Coxeter_matroids}

A Coxeter matroid we treat in this paper is a subset of a finite Coxeter group satisfying the Maximality Property (see~\cite[\S6.1.1]{bo-ge-wh03} and \cite{White96_theory} for more details on Coxeter matroids). In this section, we review definitions and properties of Coxeter matroids and define a matroid retraction of a finite Coxeter group. The symmetric group $\mathfrak{S}_n$ on $\{1,2,\dots,n\}$ is a typical example of a Coxeter group and we often use one-line notation $w(1)w(2)\cdots w(n)$ to express $w\in \mathfrak{S}_n$ throughout this paper. 

Let $W$ be a finite Coxeter group, so generators of $W$ are prescribed. We consider the metric $d$ on $W$ defined by 
\begin{equation} \label{eq:metric}
d(v,w)\coloneqq \ell(v^{-1}w)=\ell(w^{-1}v) \qquad\text{for $v, w\in W$}
\end{equation}
where $\ell(\ )$ denotes the length function on $W$. Note that the metric $d$ is invariant under the \emph{left} multiplication of $W$. For a subset $\A$ of $W$, we define 
\begin{equation*} 
d(v,\A)\coloneqq \min\{d(v,w)\mid w\in \A\}.
\end{equation*}
For each $v\in W$, there exists an element $\y\in \A$ such that $d(v,\y)=d(v,\A)$ but such $\y$ is not unique in general. However, it is unique when $\A$ is a Coxeter matroid as is explained later (Proposition~\ref{prop_mret_minimal_distance}). 

The distance $d(v,w)$ can be interpreted geometrically. As is well-known, a finite Coxeter group $W$ can be regarded as a reflection group on a vector space $V$, where the generators of $W$ act on $V$ as reflections (see \cite[Section 5.3]{Humphreys90}). Choose a point $\nu\in V$ which is not fixed by any reflection in $W$. The convex hull $\Delta_W$ of the orbit of $\nu$ under the action of $W$ is called the \textit{$W$-permutohedron} (see \cite[Section 2.4]{FominReading07_root} and \cite{HLT11_permutahedra}). We identify $w\cdot\nu$ with $w$ for each $w\in W$. Then the vertices of $\Delta_W$ are labeled by the elements in $W$ and two vertices $v$ and $w$ are joined by an edge of $\Delta_W$ if and only if $v^{-1}w$ is a simple reflection (see \cite[Lemma~2.13]{FominReading07_root}). Therefore, the distance $d(v,w)$ can be thought of as the minimum length of the paths in $\Delta_W$ connecting $v$ and $w$ through edges of $\Delta_W$. In other words, the metric $d$ is the graph metric on the graph obtained as the $1$-skeleton of $\Delta_W$. 

\begin{example}
Take $W=\mathfrak{S}_4$. Then $\Delta_{\mathfrak{S}_4}$ is the permutohedron of dimension $3$ with elements in $\mathfrak{S}_4$ as vertices, see Figure~\ref{fig:length}. If $v=1243$ and $w=3214$, then $v^{-1}w=4213$ and hence $d(v,w)=\ell(4213)=4$. The red path in Figure~\ref{fig:length} shows a minimum-length path joining $v$ and $w$ in $\Delta_{\mathfrak{S}_4}$.
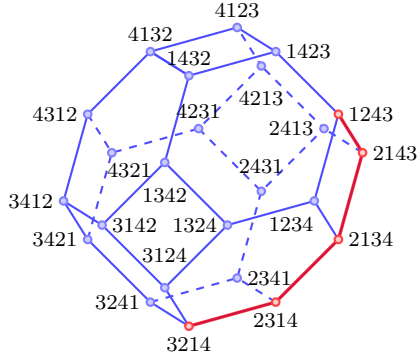
\begin{figure}[ht]
\begin{tikzpicture}[scale=5]
\tikzset{every node/.style={draw=blue!50,fill=blue!20, circle, thick, inner sep=1pt,font=\footnotesize}}
\tikzset{red node/.style = {fill=red!20!white, draw=red!75!white}}
\tikzset{red line/.style = {line width=0.3ex, red, nearly opaque}}

\coordinate (4231) at (1/3, 1/2, 1/6); 
\coordinate (2413) at (2/3, 1/2, 1/6); 
\coordinate (1243) at (5/6, 2/3, 1/2); 
\coordinate (2143) at (5/6, 1/2, 1/3); 
\coordinate (2134) at (5/6, 1/3, 1/2); 
\coordinate (1423) at (2/3, 5/6, 1/2); 
\coordinate (3142) at (1/3, 1/2, 5/6); 
\coordinate (1324) at (2/3, 1/2, 5/6); 
\coordinate (1234) at (5/6, 1/2, 2/3); 
\coordinate (1342) at (1/2, 2/3, 5/6); 
\coordinate (4123) at (1/2, 5/6, 1/3); 
\coordinate (4213) at (1/2, 2/3, 1/6); 
\coordinate (1432) at (1/2, 5/6, 2/3); 
\coordinate (4132) at (1/3, 5/6, 1/2); 
\coordinate (2314) at (2/3, 1/6, 1/2); 
\coordinate (3214) at (1/2, 1/6, 2/3); 
\coordinate (3124) at (1/2, 1/3, 5/6); 
\coordinate (3241) at (1/3, 1/6, 1/2); 
\coordinate (2341) at (1/2, 1/6, 1/3); 
\coordinate (2431) at (1/2, 1/3, 1/6);
\coordinate (3421) at (1/6, 1/3, 1/2); 
\coordinate (4321) at (1/6, 1/2, 1/3); 
\coordinate (3412) at (1/6, 1/2, 2/3); 
\coordinate (4312) at (1/6, 2/3, 1/2); 

\draw[thick, draw=blue!70] (1432)--(4132)--(4123)--(1423)--cycle;
\draw[thick, draw=blue!70] (4132)--(1432)--(1342)--(3142)--(3412)--(4312)--(4132);
\draw[dashed, thick, draw=blue!70] (4312)--(4321)--(4231)--(4213)--(4123);
\draw[dashed, thick, draw=blue!70] (3421)--(4321);
\draw[thick, draw=blue!70] (1342)--(1324)--(3124)--(3142);
\draw[dashed, thick, draw=blue!70] (4231)--(2431)--(2413)--(4213);
\draw[thick, draw=blue!70] (1423)--(1243)--(2143);
\draw[dashed, thick, draw=blue!70] (2143)--(2413);
\draw[thick, draw=blue!70] (1324)--(1234)--(1243);
\draw[thick, draw=blue!70] (1234)--(2134)--(2143);
\draw[thick, draw=blue!70] (2314)--(2134);
\draw[dashed, thick, draw=blue!70] (2314)--(2341)--(2431);
\draw[thick, draw=blue!70] (3412)--(3421)--(3241)--(3214)--(3124);
\draw[thick, draw=blue!70] (3214)--(2314);
\draw[dashed,thick, draw=blue!70] (3241)--(2341);

\draw[red line] (1243)--(2143)--(2134)--(2314)--(3214);

\node[label = {[label distance = 0cm]below left:1234}] at (1234) {};
\node[label = {[label distance = 0cm]right:1243}, red node] at (1243) {};
\node[label = {[label distance = 0cm]left:1324}] at (1324) {};
\node[label = {[label distance = 0cm]below:1342}] at (1342) {};
\node[label = {[label distance = 0cm]right:1423}] at (1423) {};
\node[label = {[label distance = -0.2cm]above:1432}] at (1432) {};
\node[label = {[label distance = 0cm]right:2134}, red node] at (2134) {};
\node [label = {[label distance = 0cm]right:2143}, red node] at (2143) {};
\node[label = {[label distance = -0.2cm]below:2314}, red node] at (2314) {};
\node[label = {[label distance = 0cm]right:2341}] at (2341) {};
\node[label = {[label distance = 0cm]left:2413}] at (2413) {};
\node [label = {[label distance = -0.1cm]above:2431}] at (2431) {};
\node [label = {[label distance = 0cm]above:3124}] at (3124) {};
\node[label = {[label distance = 0cm]right:3142}] at (3142) {};
\node[label = {[label distance = -0.2cm]below:3214}, red node] at (3214) {};
\node[label = {[label distance = 0cm]left:3241}] at (3241) {};
\node [label = {[label distance = 0cm]left:3412}] at (3412) {};
\node[label = {[label distance = 0cm]left:3421}] at (3421) {};
\node[label = {[label distance = -0.2cm]above:4123}] at (4123) {};
\node[label = {[label distance = -0.2cm]above:4132}] at (4132) {};
\node [label = {[label distance = 0cm]below:4213}] at (4213) {};
\node[label = {[label distance = -0.2cm]above:4231}] at (4231) {};
\node [label = {[label distance = 0cm]left:4312}] at (4312) {};
\node[label = {[label distance = -0.1cm]below right:4321}] at (4321) {};
\end{tikzpicture}
\caption{A minimum-length path between $1243$ and $3214$ in $\Delta_{\mathfrak{S}_4}$.}\label{fig:length}
\end{figure}
\end{example} 

For $u\in W$, $v\le^u w$ means $u^{-1}v\le u^{-1}w$ in the Bruhat order on $W$. A subset $\A$ of a finite Coxeter group $W$ is called a {\it Coxeter matroid} if it satisfies the \emph{Maximality Property}, i.e. for any $u\in W$, there is a unique element $v\in \A$ such that $w\le^u v$ for all $w\in \A$. 

\begin{example}\label{example_not_Coxeter}
Let $\A=\{213, 132\}$ be a subset of $\mathfrak{S}_3$. Since $213 \not\leq 132$ and $132 \not\leq 213$, there is no element $v \in \A$ such that $w \leq^{123} v$ for all $w \in \A$. Hence $\A$ is not a Coxeter matroid. However, one can check that $\A=\{231,321\}$ is a Coxeter matroid of $\mathfrak{S}_3$. 
\end{example}

We recall the characterization of Coxeter matroids in terms of polytopes by Gelfand-Serganova. Let $V$ be the vector space on which the generators of $W$ act as reflections. Let $\Phi$ be the set of roots of $W$. It is a subset of $V$ (see \cite[Section~5.4]{Humphreys90}). A convex polytope $\Delta$ in $V$ is called a \emph{$\Phi$-polytope} if every edge of~$\Delta$ is parallel to a root in $\Phi$. As before, take a point $\nu$ of $V$ which is not fixed by any reflection in $W$. Then, for a subset $\A$ of $W$, we define $\Delta_\A$ to be the convex hull of the $\A$-orbit $\{w\cdot \nu\mid w\in \A\}$ of the point $\nu$ in $V$, so $\Delta_{\A}$ lies in the $W$-permutohedron~$\Delta_W$. 

The following is a part of the Gelfand--Serganova theorem (see \cite[Theorem~6.3.1]{bo-ge-wh03}).

\begin{theorem}[Gelfand--Serganova]\label{thm_GS}
A subset $\A$ of a finite Coxeter group $W$ is a Coxeter matroid if and only if $\Delta_\A$ is a $\Phi$-polytope. 
\end{theorem}

\begin{remark}
The $W$-permutohedron $\Delta_W$ is a $\Phi$-polytope and moreover any root in $\Phi$ is parallel to some edge of $\Delta_W$. 
\end{remark}

\begin{example}\label{example_polytope_and_matroid}
Using Theorem~\ref{thm_GS}, one can check that subsets $\{123,213,132,312\}$ and $\{231,321\}$ of $\mathfrak{S}_3$ are Coxeter matroids while a subset $\{213,132\}$ of $\mathfrak{S}_3$ is not a Coxeter matroid since the edge joining the vertices $213$ and $132$ is not parallel to any root in $\Phi=\{\pm\alpha_1,\pm\alpha_2,\pm(\alpha_1+\alpha_2)\}$ (see Figure~\ref{fig:M} and Example~\ref{example_not_Coxeter}).
\end{example}

\begin{figure}[htb]
\centering
\begin{subfigure}[b]{.35\textwidth}
\centering
\begin{tikzpicture}[scale=.8]
\fill[fill=blue!30] (0,1)--(-{sqrt(3)}/2,1/2)--(-{sqrt(3)}/2,-1/2)--(0,-1)--cycle;
\draw[thick,gray,->] (0,0)--({sqrt(3)},-1);
\draw ({sqrt(3)},-1) node[right]{$\alpha_1$};
\draw[thick,gray,->] (0,0)--(0,2);
\draw (0,2) node[above]{$\alpha_2$};
\draw (2,0) node[right]{$s_2$};
\draw (2,0) node{$\updownarrow$};
\draw (1, {sqrt(3)+2/10}) node[right]{$s_1$};
\draw[<->] ({1+sqrt(3)/10},{sqrt(3)-1/10})--({1-sqrt(3)/10},{sqrt(3)+1/10});
\draw[thick] (1,{sqrt(3)})--(-1,-{sqrt(3)});
\draw[thick] (-1,{sqrt(3)})--(1,-{sqrt(3)});
\draw[thick] (2,0)--(-2,0);
\draw (0, 1.3) node{\footnotesize{$s_2s_1\nu$}};
\draw (0, -1.3) node{\footnotesize{$s_1\nu$}};
\draw (-1.2,-0.7) node{\footnotesize{$\nu$}};
\draw (-1.3,0.7) node{\footnotesize{$s_2\nu$}};
\filldraw[blue] (0,1) circle(2pt);
\filldraw[blue] (0,-1) circle(2pt);
\filldraw[blue] (-{sqrt(3)}/2,1/2) circle(2pt);
\filldraw[blue] (-{sqrt(3)}/2,-1/2) circle(2pt);
\draw[thick,blue] (0,1)--(-{sqrt(3)}/2,1/2)--(-{sqrt(3)}/2,-1/2)--(0,-1)--cycle;
\end{tikzpicture}
\subcaption{$\A=\{123,213,132,312\}$.}\label{fig_M2}
\end{subfigure}~
\begin{subfigure}[b]{.35\textwidth}
\centering
\begin{tikzpicture}[scale=.8]
\draw[thick,gray,->] (0,0)--({sqrt(3)},-1);
\draw ({sqrt(3)},-1) node[right]{$\alpha_1$};
\draw[thick,gray,->] (0,0)--(0,2);
\draw (0,2) node[above]{$\alpha_2$};
\draw (2,0) node[right]{$s_2$};
\draw (2,0) node{$\updownarrow$};
\draw (1, {sqrt(3)+2/10}) node[right]{$s_1$};
\draw[<->] ({1+sqrt(3)/10},{sqrt(3)-1/10})--({1-sqrt(3)/10},{sqrt(3)+1/10});
\draw[thick] (1,{sqrt(3)})--(-1,-{sqrt(3)});
\draw[thick] (-1,{sqrt(3)})--(1,-{sqrt(3)});
\draw[thick] (2,0)--(-2,0);
\draw (1.2,-0.7) node{\footnotesize{$s_1s_2\nu$}};
\draw (1.3,0.7) node{\footnotesize{$s_1s_2s_1\nu$}};
\filldraw[blue] ({sqrt(3)/2},1/2) circle(2pt);
\filldraw[blue] ({sqrt(3)/2},-1/2) circle(2pt);
\draw[thick,blue] ({sqrt(3)/2},1/2)--({sqrt(3)/2},-1/2);
\end{tikzpicture}
\subcaption{$\A=\{231,321\}$.}\label{fig_M3}
\end{subfigure}~
\begin{subfigure}[b]{.3\textwidth}
\centering
\begin{tikzpicture}[scale=.8] 
\draw[thick,gray,->] (0,0)--({sqrt(3)},-1);
\draw ({sqrt(3)},-1) node[right]{$\alpha_1$};
\draw[thick,gray,->] (0,0)--(0,2);
\draw (0,2) node[above]{$\alpha_2$};
\draw (2,0) node[right]{$s_2$};
\draw (2,0) node{$\updownarrow$};
\draw (1, {sqrt(3)+2/10}) node[right]{$s_1$};
\draw[<->] ({1+sqrt(3)/10},{sqrt(3)-1/10})--({1-sqrt(3)/10},{sqrt(3)+1/10});
\draw[thick] (1,{sqrt(3)})--(-1,-{sqrt(3)});
\draw[thick] (-1,{sqrt(3)})--(1,-{sqrt(3)});
\draw[thick] (2,0)--(-2,0);
\draw (0, -1.3) node{\footnotesize{$s_1\nu$}};
\draw (-1.3,0.7) node{\footnotesize{$s_2\nu$}};
\filldraw[red] (0,-1) circle(2pt);
\filldraw[red] (-{sqrt(3)}/2,1/2) circle(2pt);
\draw[thick,red] (-{sqrt(3)}/2,1/2)--(0,-1);
\end{tikzpicture}
\subcaption{$\A=\{213,132\}$.}\label{fig_M1}
\end{subfigure}
\caption{Examples of $\Delta_{\A}$.} \label{fig:M}
\end{figure}

Since the multiplication by the longest element of a finite Coxeter group $W$ reverses the Bruhat order on $W$, the Maximality Property is equivalent to the Minimality Property, that is, for each $u\in W$, there exists a unique minimal element in $\A$ with respect to $\leq^{u}$. We denote the unique minimal element by $\mret_\A(u)$, so 
\begin{equation} \label{eq:rcoxeter}
\text{$\mret_\A(u)\le^u w$\quad for all $w\in \A$}.
\end{equation}
Since $\mret_\A(u)=u$ for $u\in \A$, the map 
$$\mret_\A\colon W\to \A\subset W$$ is a retraction of $W$ onto $\A$, which we call a \emph{matroid retraction}. 

\begin{proposition} \label{prop_mret_minimal_distance}
If $\A$ is a Coxeter matroid of a finite Coxeter group $W$, then 
\begin{enumerate}
\item for each $u\in W$, there is a unique $\y\in \A$ such that $d(u,\y)=d(u,\A)$, and
\item $\y=\mret_{\A}(u)$.
\end{enumerate}
\end{proposition}

\begin{proof}
First we remark that if $v<^u w$, then $d(u,v)<d(u,w)$ for $u,v,w\in W$. Indeed, this follows from the following observation: 
\[
\begin{split}
v<^u w 
\Longleftrightarrow \ &u^{-1}v< u^{-1}w\\
\Longrightarrow \ & d(e,u^{-1}v)< d(e,u^{-1}w)\\
\Longleftrightarrow \ & d(u,v)< d(u,w),
\end{split}
\]
where $e$ denotes the identity element of $W$ and the last equivalence follows from the invariance of the metric $d$ under the left multiplication of $W$. 

Since $\mret_{\A}(u)$ is the unique minimal element in $\A$ with respect to the order $\leq^{u}$, it follows from the above observation that $d(u,\mret_{\A}(u))\leq d(u,w)$ for every $w\in\A$ and the equality holds only when $w=\mret_{\A}(u)$. Hence $\mret_{\A}(u)$ is the unique element in $W$ closest to $u$. This proves the proposition.
\end{proof}

\begin{example}
We observed in Example~\ref{example_not_Coxeter} or~\ref{example_polytope_and_matroid} that $\A=\{213, 132\}$ is not a Coxeter matroid of $\mathfrak{S}_3$. In this case, both $213$ and $132$ are located at the same distance from $123$, so Proposition~\ref{prop_mret_minimal_distance} (1) is not satisfied. 
\end{example}
Proposition~\ref{prop_mret_minimal_distance} (1) is a necessary condition for a subset $\A$ of $W$ to be a Coxeter matroid but it is not a sufficient condition. We discuss this issue in Section~\ref{sec:characterization}.

\section{Torus orbit closures in flag varieties} \label{sect:torus-orbit}

Let $G$ be a semisimple algebraic group, $B$ a Borel subgroup of $G$, and $T$ a maximal torus of $G$ contained in $B$. The left multiplication of $T$ on $G$ induces an action of $T$ on the flag variety $G/B$ and the $T$-fixed point set $(G/B)^T$ of $G/B$ can be identified with the Weyl group $W$ of $G$. 
For a point $x\in G/B$, the closure $Y\coloneqq \overline{T\cdot x}$ of its $T$-orbit $T\cdot x$ is a (possibly non-normal) toric variety in $G/B$. 
In this section, we define a geometric retraction $\gret_Y$ of the Weyl group $W$ onto the $T$-fixed point set $Y^T \subset W$ of $Y$ by using the Orbit-Cone correspondence in toric variety. 
It is known that $Y^T$ is a Coxeter matroid of $W$. 
We show that $\gret_Y=\mret_{Y^T}$ for every torus orbit closure $Y$ in $G/B$. 

We think of the real Lie algebra $\mathfrak{t}_\R$ of the torus $T$ as $\Hom(\C^{\ast},T)\otimes \R$ and the lattice $\mathfrak{t}_\Z$ of $\mathfrak{t}_\R$ as $\Hom(\C^{\ast},T)$, where $\Hom(\C^*,T)$ denotes the group of algebraic homomorphisms from $\C^*=\C\backslash\{0\}$ to $T$. The vector space $\mathfrak{t}_\R^{\ast}$ dual to $\mathfrak{t}_\R$ can be thought of as $\Hom(T,\C^{\ast})\otimes \R$ and the set $\Phi$ of roots of $G$ is a finite subset of $\Hom(T,\C^{\ast})=\mathfrak{t}_\Z^{\ast}$. The Weyl group $W$ of $G$ acts on $\mathfrak{t}_\R$ as the adjoint action and on its dual space $\mathfrak{t}^{\ast}_\R$ as the coadjoint action, i.e. 
\[
\text{$(w\cdot f)(x)\coloneqq f(\Ad_{w^{-1}}(x))$\quad for $w\in W$, $f\in \mathfrak{t}^{\ast}_\R$ and $x\in \mathfrak{t}_\R$.}
\]
The Borel subgroup $B$ determines the set $\Phi^+$ of positive roots and we have
\begin{equation} \label{eq:tangent}
\quad T_{u}(G/B)=\bigoplus_{\alpha\in \Phi^+}\mathfrak{g}_{-u(\alpha)}\quad \text{as $T$-modules for $u\in W$,}
\end{equation}
where $\mathfrak{g}_{\alpha}$ denotes the eigensubspace of $\mathfrak{g}$ for $\alpha\in \Phi$ (see, for example,~\cite[\S3]{GHZ06_GKM}). 

For each $u\in W$, we define 
\begin{equation} \label{eq:cone}
C(u)=\{ \lambda\in \mathfrak{t}_\R\mid \langle u(\alpha),\lambda\rangle \le 0 \quad\text{for all simple roots $\alpha$}\}
\end{equation}
where $\langle \ ,\ \rangle$ denotes the natural pairing between $\mathfrak{t}_\R^{\ast}$ and $\mathfrak{t}_\R$.
The interiors of the cones $C(u)$ above form the Weyl chambers. The identity in~\eqref{eq:tangent} implies that for any $\lambda_u\in \Int (C(u))\cap \mathfrak{t}_\Z$, we have 
\begin{equation} \label{eq:lambda}
(G/B)^{\lambda_u(\C^{\ast})}=(G/B)^T.
\end{equation}

For each $w\in W=(G/B)^T$, we choose an element $\lambda_w\in \Int (C(w))\cap \mathfrak{t}_\Z$ and define 
\[
S_w\coloneqq \left\{ x\in G/B \,\middle|\, \lim_{t\to 0} \lambda_w(t)\cdot x=w \right\},
\]
which is independent of the choice of $\lambda_w$. Then $S_w$ is a $T$-invariant affine open subset of $G/B$ and isomorphic to $T_w(G/B)$ as a $T$-variety (see~\cite{BB73_some}). 

\begin{proposition} \label{lemm:limit_Y}
Let $x$ be a point of $G/B$ and $Y=\overline{T\cdot x}$. For any $u\in W$ and $\lambda_u \in \Int (C(u))\cap \mathfrak{t}_\Z$, the limit point ${\lim_{t\to 0}\lambda_u(t)\cdot x}$ is an element of~$Y^T$ depending only on $u$ and $Y$. Furthermore, if $u\in Y^T$, then $\lim_{t\to 0}\lambda_u(t)\cdot x=u$.
\end{proposition}

\begin{proof}
Since $Y$ is closed, the limit point ${\lim_{t\to 0}\lambda_u(t)\cdot x}$ belongs to $Y$ and clearly remains fixed under the action of $\lambda_u(\C^{\ast})$. Therefore, the limit point is indeed in $Y^T$ by~\eqref{eq:lambda}. Denote the limit point by $w$. Since $\lambda_u(t)\cdot x \in S_{w}$ and $S_{w}$ is $T$-invariant, $x$ belongs to $S_{w}$. Since $S_w$ is isomorphic to $T_w(G/B)$ as a $T$-variety, it follows from \eqref{eq:tangent} that ${\lim_{t\to 0}\lambda_u(t)\cdot x}$ is independent of the choice of $\lambda_u\in \Int (C(u))\cap \mathfrak{t}_\Z$. 

If $u\in Y^T$, then $x$ belongs to $S_u$ because otherwise $Y=\overline{T\cdot x}$ does not contain $u$ (note that $S_u$ is a $T$-invariant open subset of $G/B$).  Therefore, $\lim_{t\to 0}\lambda_u(t)\cdot x=u$.
\end{proof}

By Proposition~\ref{lemm:limit_Y}, the map $\gret_{Y}\colon W\to Y^T\subset W$ defined by 
\begin{equation} \label{eq:limit_Y}
\gret_{Y}(u)\coloneqq \lim_{t\to 0}\lambda(t)\cdot x
\end{equation}
is a retraction of $W$ onto $Y^T$, which we call a \emph{geometric retraction}. 

\begin{example}\label{example_fan_Y_SL3}
Take a point 
$x=\begin{pmatrix}1&1&0\\1&0&1\\1&0&0 
\end{pmatrix}B \in \mathrm{SL}_3(\C)/B$ and consider $Y=\overline{T\cdot x}$. One can check that $Y^T=\{123,132,213,312\}$.
Let us choose an element $\lambda = (\lambda_1, \lambda_2, \lambda_3) \in \Int(C(231)) \cap \mathfrak{t}_{\Z}$, that is $\lambda_2 < \lambda_3 < \lambda_1$. Then we have that
\[
\begin{split}
\lambda(t) \cdot \begin{pmatrix}
1 & 1 & 0 \\ 1 & 0 & 1 \\ 1 & 0 & 0
\end{pmatrix}B
&= \begin{pmatrix}
t^{\lambda_1} & 0 & 0 \\
0 & t^{\lambda_2} & 0 \\
0 & 0 & t^{\lambda_3}
\end{pmatrix} \cdot
\begin{pmatrix}
1 & 1 & 1 \\ 1 & 0 & 0 \\ 1 & 0 & 1
\end{pmatrix}B \\
&= \begin{pmatrix}
t^{\lambda_1} & t^{\lambda_1} & t^{\lambda_1} \\ t^{\lambda_2} & 0 & 0\\ t^{\lambda_3} & 0 & t^{\lambda_3}
\end{pmatrix}B \\
&= \begin{pmatrix}
t^{\lambda_1 - \lambda_2} & 1 & t^{\lambda_1 - \lambda_3} \\
1 & 0 & 0 \\
t^{\lambda_3 - \lambda_2} & 0 & 1
\end{pmatrix}B 
\stackrel{t \to 0}{\longrightarrow}
\begin{pmatrix}
0 & 1 & 0 \\ 1 & 0 & 0 \\ 0 & 0 & 1
\end{pmatrix}B.
\end{split}
\]
Therefore we get $\gret_Y(231)=213$. By a similar computation, we obtain Table~\ref{table_ug_1}.
\begin{table}[htb]
\centering
\begin{tabular}{c|cccccc}
\toprule 
$u$ & $123$ & $213$ & $231$ & $321$ & $312$ & $132$ \\
\midrule 
$\gret_Y(u)$ for $Y$ & $123$ & $213$ & $213$ & $312$ & $312$ & $132$ \\
\bottomrule
\end{tabular}

\smallskip
\caption{$\gret_Y(u)$ for $Y = \overline{T\cdot x}$ in Example~\ref{example_fan_Y_SL3}.}\label{table_ug_1}
\end{table}
\end{example}

We make a remark about the fan of $Y$. The following corollary follows from Proposition~\ref{lemm:limit_Y} and the Orbit-Cone correspondence of toric varieties (see \cite[Proposition 3.2.2 \& Theorem 3.A.5]{CLS}). 

\begin{corollary}\textup{(}cf.~\cite[Corollary 3.7]{le-ma18}\textup{)} \label{coro:fan_of_Y}
The maximal cone $C_Y(y)$ corresponding to $y\in Y^T$  in the fan of (the normalization of) $Y$ is given by $\bigcup_{u\in(\gret_Y)^{-1}(y)}C(u)$. 
\end{corollary}

\begin{remark}\label{rmk_normal}
\begin{enumerate}
\item Since the action of $T$ on $Y$ is not effective, the ambient space of the fan of $Y$ is the quotient of $\mathfrak{t}_{\R}$ by the subspace $\Hom(\C^{\ast}, T_Y) \otimes \R$, where $T_Y$ is the toral subgroup of $T$ which fixes $Y$ pointwise. Therefore, to be precise, we need to project the cones $C_Y(y)$ to this quotient space in the corollary above. 

\item When $G$ is of type $A_n$, $D_4$, or $B_2$, every $T$-orbit closure in $G/B$ is normal (\cite[Proposition~4.8]{ca-ku2000}) while when $G = G_2$, non-normal torus orbit closures exist (see~\cite[Example~6.1]{ca-ku2000} and references therein). 
\end{enumerate}
\end{remark}

Now we reformulate the geometric retraction $\gret_Y$ using a Bruhat decomposition of $G/B$. We thank a referee for pointing out this reformulation. It makes the meaning of the geometric retraction more transparent and substitutes our original discussion using GKM graph. 

For $u \in W$, we set $B_u \coloneqq u B^{-}u^{-1}$, where $B^-$ is the opposite Borel subgroup $w_0Bw_0$ and $w_0$ is the longest element of $W$. The Lie algebra of the Borel subgroup $B_u$ is as follows:
\begin{equation}\label{eq_Lie_Bu}
\Lie(B_u) = \mathfrak{t} \oplus \bigoplus_{\alpha \in \Phi^+} \frak{g}_{-u(\alpha)}. 
\end{equation}

With respect to $B_u$, we obtain the following Bruhat decomposition:
\begin{equation} \label{eq:decomposition_of_G/B}
G/B = \bigsqcup_{w \in W} B_u \cdot w B/B
\end{equation}
and set 
\begin{equation} \label{eq:Auw}
A^u_w \coloneqq B_u \cdot w B/B=u\cdot B^{-}u^{-1}wB/B.
\end{equation}
Note that $A^{w_0}_w$ is the Schubert cell $BwB/B$ and $A^e_w$ is the opposite Schubert cell $B^- w B/B$, where $e$ denotes the identity element of $W$. Similarly to (opposite) Schubert cells, we have 
\begin{equation}\label{eq_closure_of_Auw}
\overline{A^u_w} = \bigsqcup_{ w \leq^u v \leq^u u w_0} A^u_v \quad\text{and hence $(\overline{A^u_w})^T = \{ v \in W \mid w \leq^u v \leq^u uw_0\}$. }
\end{equation}
Indeed, since $\overline{B^- wB/B}$ is the disjoint union of opposite Schubert cells indexed by elements $z\in W$ satisfying $w \leq z \leq w_0$, it follows from \eqref{eq:Auw} that 
\[
\begin{split}
\overline{A^u_w} &= u \cdot \overline{B^{-}u^{-1}wB/B} \\
&= u \bigsqcup_{ u^{-1}w \leq z \leq w_0} B^- z B/B \\
&= \bigsqcup_{u^{-1}w\le u^{-1}(uz) \le u^{-1}(uw_0)}u\cdot B^-u^{-1}(uz)B/B\\
&= \bigsqcup_{w \leq^u v \leq^u uw_0} A^u_v,
\end{split}
\]
which shows \eqref{eq_closure_of_Auw}. 

Recall that for each root $\alpha\in \Hom(T, \C^{\ast})$, there exists a root homomorphism $x_{\alpha} \colon \C \to G$ satisfying
\begin{equation}\label{eq_root_homomorphism}
s x_{\alpha}(a) s^{-1} = x_{\alpha}(\alpha(s) a)
\end{equation}
for any $s \in T$ and $a \in \C$, and moreover, $d x_{\alpha} \colon \C \stackrel{\sim}{\longrightarrow} \mathfrak{g}_{\alpha}$. We call the image of $x_{\alpha}$ \textit{the root subgroup} of $G$ and denote it by $U_{\alpha}$.

\begin{proposition}\label{prop1}
Let $x$ be a point of $G/B$ and $Y=\overline{T\cdot x}$. Then $x \in A^u_w$ if and only if $\gret_Y(u) = w$.
\end{proposition}
\begin{proof}
Since the flag variety $G/B$ is the disjoint union of $A^u_w$ over $w\in W$ by \eqref{eq:decomposition_of_G/B} and \eqref{eq:Auw}, it suffices to prove the `only if' part.

Let $\Phi^+ = \{\beta_1,\dots,\beta_N\}$ and consider $\mathfrak{u}_u = \bigoplus_{\beta_j \in \Phi^+} \frak{g}_{-u(\beta_j)} \subset \mathfrak{g}$. The unipotent subgroup $U_u:=\exp(\mathfrak{u}_u)$ of $G$ is given by 
\[
U_u = \exp(\mathfrak{u}_u)=U_{-u(\beta_1)}U_{-u(\beta_2)}\cdots U_{-u(\beta_N)}.
\] 
By~\eqref{eq_Lie_Bu}, we have $B_u = U_u \rtimes T$, so that $A^u_w = U_u wB/B$. Therefore, an element $x \in A^u_w$ can be written as $x = u_1 \cdots u_N wB$ for some $u_j \in U_{-u(\beta_j)}$.

For $\lambda_u \in \Int(C(u)) \cap \mathfrak{t}_{\Z}$ and $t \in \C^{\ast}$, we get $\alpha(\lambda_u(t)) = t^{\langle \alpha, \lambda_u \rangle}$. Accordingly, for any positive root $\alpha \in \Phi^+$ and $a \in \C$, we
have that
\begin{equation}\label{eq_lim_lambda_us}
\begin{split}
\lim_{t \to 0}
\lambda_u(t) x_{-u(\alpha)}(a) \lambda_u(t)^{-1} 
&= \lim_{t \to 0} x_{-u(\alpha)} ( -u(\alpha)(\lambda_u(t)) a) \quad \text{(by~\eqref{eq_root_homomorphism})}\\
&= \lim_{t \to 0} x_{-u(\alpha)}(t^{\langle -u(\alpha), \lambda_u \rangle} a) \\
&= x_{-u(\alpha)}(0) = I, 
\end{split}
\end{equation}
where $I$ is the identity element of $G$ and the third equality above follows from the definition of $C(u)$ in~\eqref{eq:cone}. Noting that $wB$ is a $T$-fixed point and using~\eqref{eq_lim_lambda_us}, we get the following for $x = u_1 \cdots u_N wB \in A^u_w$: 
\begin{equation*}
\begin{split}
\lim_{t \to 0} \lambda_u(t) \cdot x 
&= \lim_{t \to 0} \lambda_u(t) \cdot u_1 \cdots u_N wB \\
&= \lim_{t \to 0} (\lambda_u(t) u_1 \lambda_u(t)^{-1})
(\lambda_u(t) u_2 \lambda_u(t)^{-1})\cdots(\lambda_u(t) u_N \lambda_u(t)^{-1}) \lambda_u(t) \cdot w B \\
&= \lim_{t \to 0} (\lambda_u(t) u_1 \lambda_u(t)^{-1})
(\lambda_u(t) u_2 \lambda_u(t)^{-1})\cdots(\lambda_u(t) u_N \lambda_u(t)^{-1}) \cdot w B \\
&= I \cdot wB = wB.
\end{split}
\end{equation*}
This proves that if $x \in A^u_w$, then $\gret_Y(u) = w$. 
\end{proof}

\begin{example}
Let $x = \begin{pmatrix} 1 & 1 & 0 \\ 1 & 0 & 1 \\ 1 & 0 & 0 \end{pmatrix}B\in \SL_3(\C)/B$. 
Take $u = 231\in \mathfrak{S}_3$. Then the corresponding Borel subgroup $B_u$ is given by 
\[
B_u = u B^- u^{-1} 
= \begin{pmatrix}
0 & 0 & 1 \\1 & 0 & 0 \\ 0 & 1 & 0
\end{pmatrix}
\begin{pmatrix} \star & 0 & 0 \\ * & \star & 0 \\ * & * & \star \end{pmatrix}
\begin{pmatrix} 0 & 1 & 0 \\ 0 & 0 & 1 \\ 1 & 0 & 0 \end{pmatrix}
= \begin{pmatrix}
\star & * & * \\ 0 & \star & 0 \\ 0 & * & \star
\end{pmatrix}.
\]
Here, $* \in \C$ and $\star \in \C^{\ast}$.
For $w = 213\in\mathfrak{S}_3$, the cell $A^u_w$ is given by 
\[
A^u_w = B_u w B/B = \begin{pmatrix}
\star & * & * \\ 0 & \star & 0 \\ 0 & * & \star
\end{pmatrix} \begin{pmatrix}
0 & 1 & 0 \\ 1 & 0 & 0 \\ 0 & 0 & 1
\end{pmatrix} B 
= \begin{pmatrix}
* & \star & * \\ \star & 0 & 0 \\ * & 0 & \star 
\end{pmatrix}B
= \begin{pmatrix}
* & 1 & 0 \\ 1 & 0 & 0 \\ * & 0 & 1
\end{pmatrix} B.
\]
Since 
\begin{equation}\label{eq_computation_x_in_Auw}
\begin{pmatrix}
1 & 1 & 0 \\ 1 & 0 & 0 \\ 1 & 0 & 1
\end{pmatrix}B
= \begin{pmatrix}
1 & 1 & -1 \\ 1 & 0 & -1 \\ 1 & 0 & 0
\end{pmatrix}B 
= \begin{pmatrix}
1 & 1 & 0 \\ 1 & 0 & 1 \\ 1 & 0 & 0
\end{pmatrix}B,
\end{equation}
we see that the element $x$ is contained in $A^{231}_{213}$. 

Take $\lambda_u = (\lambda_1, \lambda_2, \lambda_3) \in \Int(C(u))\cap \mathfrak{t}_{\Z}$, i.e. $\lambda_2 < \lambda_3 < \lambda_1$. 
Using \eqref{eq_computation_x_in_Auw}, we obtain
\[
x = \begin{pmatrix}
1 & 1 & 0 \\ 1 & 0 & 0 \\ 1 & 0 & 1
\end{pmatrix}B = \underbrace{\begin{pmatrix}
1 & 1 & 0 \\ 0 & 1 & 0 \\ 0 & 1 & 1
\end{pmatrix}}_{\in B_u}
\begin{pmatrix}
0 & 1 & 0 \\ 1 & 0 & 0 \\ 0 & 0 & 1
\end{pmatrix}B.
\]
Then we have 
\[
\begin{split}
\lim_{t \to 0} \lambda_u(t)\cdot x 
&= \lim_{t \to 0} \begin{pmatrix}
t^{\lambda_1} & 0 & 0 \\ 0 & t^{\lambda_2} & 0 \\ 0 & 0 & t^{\lambda_3}
\end{pmatrix}
\begin{pmatrix}
1 & 1 & 0 \\ 0 & 1 & 0 \\ 0 & 1 & 1
\end{pmatrix}
\begin{pmatrix}
0 & 1 & 0 \\ 1 & 0 & 0 \\ 0 & 0 & 1
\end{pmatrix}B \\
&= \lim_{t \to 0}
\begin{pmatrix}
t^{\lambda_1} & t^{\lambda_1} & 0 \\
0 & t^{\lambda_2} & 0 \\
0 & t^{\lambda_3} & t^{\lambda_3} 
\end{pmatrix}\begin{pmatrix}
0 & 1 & 0 \\ 1 & 0 & 0 \\ 0 & 0 & 1
\end{pmatrix}B \\
&= \lim_{t \to 0}
\begin{pmatrix}
1 & t^{\lambda_1 - \lambda_2} & 0 \\
0 & 1 & 0 \\ 0 & t^{\lambda_3 - \lambda_2} & 1 
\end{pmatrix}
\begin{pmatrix}
t^{\lambda_1} & 0 & 0 \\ 0 & t^{\lambda_2} & 0 \\ 0 & 0 & t^{\lambda_3}
\end{pmatrix}\begin{pmatrix}
0 & 1 & 0 \\ 1 & 0 & 0 \\ 0 & 0 & 1
\end{pmatrix}B \\
&= \lim_{t \to 0} 
\begin{pmatrix}
1 & t^{\lambda_1 - \lambda_2} & 0 \\
0 & 1 & 0 \\ 0 & t^{\lambda_3 - \lambda_2} & 1 
\end{pmatrix}
\begin{pmatrix}
0 & 1 & 0 \\ 1 & 0 & 0 \\ 0 & 0 & 1
\end{pmatrix}B \\ 
&= \begin{pmatrix}
0 & 1 & 0 \\ 1 & 0 & 0 \\ 0 & 0 & 1
\end{pmatrix}B.
\end{split}
\] 
Accordingly, for $Y = \overline{T\cdot x}$, we obtain that $\gret_Y(231) = 213$. 
\end{example}

The $T$-fixed point set $Y^T$ of a $T$-orbit closure $Y$ in the flag variety $G/B$ is known to be a Coxeter matroid of the Weyl group $W$ of $G$ (\cite{GelfandSerganova87}), so we have the matroid retraction $\mret_{Y^T}\colon W\to Y^T\subset W$. 

\begin{theorem}\label{theo:torus-coxeter}
Let $G$ be a semisimple algebraic group over $\C$, $B$ a Borel subgroup of $G$, and $T$ a maximal torus of $G$ contained in $B$. Then $\gret_Y=\mret_{Y^T}$ for any $T$-orbit closure $Y$ in $G/B$. 
\end{theorem}
\begin{proof}
Let $x$ be a point of $G/B$ such that $Y=\overline{T\cdot x}$.
For $u\in W$, let $\mathcal{R}_{Y}^{g}(u)=w$. By Proposition~\ref{prop1}, $T\cdot x \subseteq A^{u}_{w}$ and hence $Y=\overline{T\cdot x}\subseteq \overline{A^{u}_{w}}$. Therefore, we have
\[
Y^T\subset \overline{A^u_w}^T=\{ v\in W\mid w\le^u v \le^u uw_0\},
\]
where the equation above follows from \eqref{eq_closure_of_Auw}. Hence $w$ is the unique minimal element in $Y^T$ with respect to the ordering $\leq^u$. Therefore, $\gret_Y=\mret_{Y^{T}}$.
\end{proof}

\begin{remark}
The above proof gives an alternative proof of $Y^T$ being a Coxeter matroid. 
\end{remark}

\section{Representability of Coxeter matroids} \label{sect:representability}

When $W$ is the Weyl group of a semisimple algebraic group $G$, a Coxeter matroid~$\A$ of $W$ is said to be \emph{representable} if $\A$ can be realized as the $T$-fixed point set of a $T$-orbit closure in $G/B$. See~\cite[\S 1.7.5, \S 3.6.2, \S 3.10.3]{bo-ge-wh03}. In this section we discuss the representability of Coxeter matroids when $G=\SL_n(\C)$, so $W=\mathfrak{S}_n$. 

A computer check shows that any Coxeter matroid of $\mathfrak{S}_n$ is representable when $n\le 4$. On the other hand, there exists a non-representable Coxeter matroid of $\mathfrak{S}_n$ when $n=7$, which we shall explain in the following. 

Let $x\in\SL_n(\C)/B$ and choose a representative $\tilde{x}\in \SL_n(\C)$ of $x$. For a sequence $\underline{i} = (i_1,\dots,i_d)$ such that $1 \leq i_1 < i_2 < \cdots < i_d \leq n$, we define $p_{\underline{i}}(\tilde{x})$ to be the $d \times d$ minor of $\tilde{x}$ with row indices $i_1,\dots,i_d$ and the column indices $1,\dots,d$. Since the non-vanishing of $p_{\underline{i}}(\tilde{x})$ is independent of the choice of the representative $\tilde{x}$ of $x$, we denote the following set by $I_d(x)$: 
\[
I_d(x) := \{ \underline{i} = (i_1,\dots,i_d) \mid p_{\underline{i}}(\tilde{x}) \neq 0\} \quad 1 \leq d \leq n.
\]

\begin{proposition}[{\cite[Proposition~1 in \S5.2]{GelfandSerganova87}}]\label{proposition_GS_fixed_points}
For an element $x \in \SL_n(\C)/B$, we have that
\[
(\overline{T\cdot x})^T = \{w \in \mathfrak{S}_n \mid \{w(1),\dots,w(d)\}\!\!\uparrow~ \in I_d(x)\quad \text{for all } 1\leq d \leq n\}.
\]
Here, $\{a_1,\dots,a_d\}\!\!\uparrow$ is the ordered $d$-tuple obtained from $\{a_1,\dots,a_d\}$ by arranging its elements in ascending order.
\end{proposition}

\begin{example}
When $x$ is the element of $\SL_3(\C)/B$ in Example~\ref{example_fan_Y_SL3}, we have that
\begin{gather*}
I_1(x) = \{(1), (2), (3) \}, \quad I_2(x) = \{(1,2),(1,3) \}, \quad I_3(x) = \{(1,2,3)\}, \text{ and }\\
\{123,132, 213, 312\} = \{w \in \mathfrak{S}_3 \mid \{w(1),\dots,w(d)\}\!\!\uparrow~ \in I_d(x) \quad \text{ for all }1 \leq d \leq 3\}.
\end{gather*}
\end{example}

Using Proposition~\ref{proposition_GS_fixed_points}, one can find a non-representable Coxeter matroid of $\mathfrak{S}_7$ as follows. 
Let 
\[
A = \{\{1,2,4\}, \{1,3,5\}, \{1,6,7\}, \{2,3,6\}, \{2,5,7\}, \{3,4,7\}, \{4,5,6\}\}
\] 
be a collection of subsets of $\{1,2,\ldots,7\}$. Each subset corresponds to one of the seven lines on the Fano plane in Figure~\ref{fig_Fano}. Define a subset $\A$ of $\mathfrak{S}_7$ by
\[
\A = \{w \in \mathfrak{S}_7 \mid \{w(1), w(2), w(3)\} \notin A\}.
\] 
Then $\A$ is a Coxeter matroid obtained from the Fano matroid using Higgs lifts (see~\cite[\S1.7]{bo-ge-wh03} for the definition of Higgs lift).
Suppose that $\A$ is the $T$-fixed point set of a $T$-orbit closure in $\SL_7(\C)/B$. Then, by Proposition~\ref{proposition_GS_fixed_points}, there is an element $x$ in $\SL_7(\C)/B$ such that $$\A = \{w\in \mathfrak{S}_7\mid \{w(1),\ldots,w(d)\}\!\!\uparrow \in I_d(x)\text{ for all }1\leq d\leq 7\}.$$ However, it is known in~\cite[\S16]{Whitney} and easy to check that there is no $7\times 3$ matrix of rank~$3$ whose three rows $v_{j_1}, v_{j_2}, v_{j_3}$ are linearly independent if and only if $\{j_1, j_2, j_3\} \notin A$. Therefore, $\A$ cannot be obtained as $Y^T$ for a $T$-orbit closure $Y$ in $\SL_7(\C)/B$.
\begin{figure}[h]
\begin{tikzpicture}
\draw (0:0) circle (1cm);
\node[fill=white, draw, circle] at (210:2) (1) {$1$};
\node[fill=white, draw, circle] at (90:2) (2) {$2$};
\node[fill=white, draw, circle] at (330:2) (3) {$3$};
\node[fill=white, draw, circle] at (150:1) (4) {$4$};
\node[fill=white, draw, circle] at (270:1) (5) {$5$};
\node[fill=white, draw, circle] at (30:1) (6) {$6$} ;
\node[fill=white, draw, circle] at (0:0) (7) {$7$};
\draw (1)--(5);
\draw (5)--(3);
\draw (3)--(6);
\draw (6)--(2);
\draw (2)--(4);
\draw (4)--(1);
\draw(1)--(7);
\draw (7)--(6);
\draw(5)--(7);
\draw (7)--(2);
\draw (3)--(7);
\draw (7)--(4);
\end{tikzpicture}
\caption{The Fano plane.}\label{fig_Fano}
\end{figure}

When $G$ is of Lie type $A$, torus orbit closures associated to Schubert varieties
and Richardson varieties were studied in \cite{le-ma18} and \cite{LMP_BIP}, respectively.

\begin{remark}
Let $v$ and $w$ be elements in $\mathfrak{S}_{n}$ with $v\leq w$ in Bruhat order. Then we have the Richardson variety $X^v_w\coloneqq X_w\cap w_0 X_{w_0v}$ in the flag variety $\SL_{n}(\C)/B$, where $X_w \coloneqq \overline{BwB/B}$ is the Schubert variety associated to $w \in \mathfrak{S}_{n}$ and $w_0$ is the longest element of $\mathfrak{S}_n$. Note that $X^v_w$ is $T$-invariant and $(X^v_w)^T$ is the Bruhat interval 
\[
[v,w] \coloneqq \{z\in \mathfrak{S}_{n}\mid v\leq z\leq w\}. 
\]
One can see that there exists a point $x\in X^v_w$ such that 
\[
(\overline{T\cdot x})^T= (X^v_w)^{T}(=[v,w]).
\]
The existence of such a point $x$ is proven in~\cite[Proposition~3.8]{le-ma18} when $X^v_w$ is a Schubert variety, i.e. $v=e$ and a similar argument works for any Richardson variety $X^v_w$. Therefore, every Bruhat interval is a representable Coxeter matroid.
\end{remark}

\section{Algebraic retractions} \label{sect:Retraction on the Weyl group}

In this section, we define the third retraction $\aret_\A$ of $W$ onto $\A$ algebraically when $W$ is a product $\prod_{j=1}^m W_j$ of Weyl groups $W_j$'s of classical Lie types and $\A=\prod_{j=1}^m\A_j$ where each $\A_j$ is an \emph{arbitrary} subset of $W_j$. In general, $\aret_\A(u)$ is not necessarily an element of $\A$ closest to $u\in W$, but it tunrs out that $\aret_{\A}=\mret_\A$ when $\A$ is a Coxeter matroid, so that $\aret_{\A}$ produces the closest elements when $\A$ is a Coxeter matroid by Proposition~\ref{prop_mret_minimal_distance}.

The Weyl group $W$ of classical Lie type is of the following form: 
\[
W=\begin{cases} \mathfrak{S}_n \qquad &\text{if $W$ is of type $A_{n-1}$},\\
(\Z/2\Z)^n\rtimes \mathfrak{S}_n \qquad &\text{if $W$ is of type $B_n$ or $C_n$},\\
(\Z/2\Z)^{n-1}\rtimes \mathfrak{S}_n\qquad &\text{if $W$ is of type $D_n$}.
\end{cases}
\]
We denote the set $\{\bar{1},\dots,\bar{n}\}$ by $[\overline{n}]$ and regard $\bar{\bar{i}}=i$.
In each type we will use one-line notation for $u\in W$, i.e.
\[
u=u(1)u(2)\cdots u(n)
\]
where $u(i)\in [n]\cup[\bar{n}]$ and $u(1)u(2)\cdots u(n)$ is a permutation on $[n]$ if we forget the bars. 
There is no bar in type $A$ and the number of bars in $u(1),\dots,u(n)$ is even (possibly zero) in type $D$. 
In types $B$, $C$ and $D$, we have $u(\bar{i})=\overline{u(i)}$.

For $u\in W$, let us define a linear order $\prec^u$ on the set $[n]\cup[\overline{n}]$ by
\begin{equation}
u(1)\prec^u \dots \prec^u u(n)\prec^u u(\overline{n})\prec^u\dots \prec^u u(\overline{1}).
\end{equation}
This induces a $u$-lexicographic order $\prec^u_{\mathrm{lex}}$ on the set of words of length $n$ in the alphabet $[n]\cup [\overline{n}]$. Then we obtain a linear order $\prec^u$ on $W$, where $v\prec^u w$ if and only if $v(1)\dots v(n)\prec^u_{\mathrm{lex}} w(1)\dots w(n)$. Note that $u$ is the minimal element of $W$ with respect to $\prec^u$.

\begin{definition} \label{defi:algebraic_retraction}
Let $W$ be a Weyl group of classical Lie type and $\mathcal{M}$ an \emph{arbitrary} subset of $W$. For each $u\in W$, we define $\aret_\A(u)$ as the $u$-minimal element of $\mathcal{M}$ with respect to the ordering $\prec^u$. Then the map
$$\aret_\A\colon W\to \A\, (\subset W)$$ is a retraction of $W$ onto $\A$, which we call an \emph{algebraic retraction}.
\end{definition}

\begin{remark}
Originally, we defined $\aret_\A$ following the idea of the retraction introduced in \cite[Definition 3.3]{le-ma18}. The definition presented here is due to a referee. We thank him/her for the simplification of our original definition of $\aret_A$. 
\end{remark}

\begin{example}\label{ex:alg_retraction2}
\begin{enumerate}
\item We take a subset $\A=\{1423, 1432, 2413, 3412\}$ of $\mathfrak{S}_4$. For $u=1324$, we have a linear order
$$1\prec^u 3 \prec^u 2\prec^u 4.$$
Then 
$$1432 \prec^u_{\mathrm{lex}}1423\prec^u_{\mathrm{lex}} 3412\prec^u_{\mathrm{lex}} 2413.$$
Hence $\aret_\A(u)=1432.$
\item We take a subset $\A=\{1\bar{4}23, 14\bar{3}\bar{2}, 2413, \bar{3}\bar{4}1\bar{2}\}$ of $(\Z/2\Z)^4 \rtimes \mathfrak{S}_4$. Then for $u=\bar{2}3\bar{1}4$, we have a linear order 
$$\bar{2} \prec^u 3 \prec^u \bar{1} \prec^u 4 \prec^u \bar{4} \prec^u 1 \prec^u \bar{3} \prec^u 2.$$
Then 
$$14\bar{3}\bar{2} \prec^u_{\mathrm{lex}} 1\bar{4}23\prec^u_{\mathrm{lex}} \bar{3}\bar{4}1\bar{2}\prec^u_{\mathrm{lex}} 2413,$$
so $\aret_\A(\bar{2}3\bar{1}4)=14\bar{3}\bar{2}$. 
\end{enumerate}
\end{example} 

Note that it follows from the definitions of $<^u$ and $\prec^u$ on $W$ that
\begin{itemize}
\item $v<^u w$ if and only if $u^{-1}v < u^{-1}w$ in Bruhat order, and 
\item $v\prec^u w$ if and only if $u^{-1}v\prec^{e} u^{-1}w$.
\end{itemize}

\begin{lemma}\label{lem-lex}
Let $W$ be a Weyl group of classical Lie type and $u,v,w$ elements in~$W$. Then $v\prec^u w$ if $v<^u w$.
\end{lemma}
\begin{proof}
It is enough to check that $v<w$ implies $v\prec^{e} w$. However, this claim immediately follows from the following fact:
\[
v \leq w \text{ if and only if } \{v(1),\dots,v(d)\}\!\!\uparrow~\leq \{ w(1),\dots,w(d) \}\!\!\uparrow \quad \text{for }1 \leq \forall d \leq n-1
\]
(see \cite[Section 3]{BL}). 
\end{proof}

If $W$ is a product $\prod_{j=1}^m W_j$ of Weyl groups $W_j$ of classical Lie types and $\A=\prod_{j=1}^m \A_j$, where $\A_j$ is an arbitrary subset of $W_j$, then we define an algebraic retraction of $W$ onto $\A$ by applying the algebraic retraction to each $\A_j\subset W_j$, i.e. for $u = (u_1,\dots,u_m) \in W$, we define 
\begin{equation}\label{equation_aret_product_fixed_points}
\aret_{\A}(u) \coloneqq (\aret_{\A_1}(u_1), \dots, \aret_{\A_m}(u_m)) \in \A=\prod_{j=1}^m\A_j.
\end{equation}

\begin{lemma} \label{lemm:re}
Let $W$ and $\A$ be as above. 
If $\A$ has a unique minimal element $v$ in the Bruhat order, i.e. $v\le w$ for all $w\in \A$, then $\aret_\A(e)=v$. 
\end{lemma} 

\begin{proof} 
For $v=(v_1,\dots,v_m)$ and $w=(w_1,\ldots,w_m)$ in $W$, we have that $v\leq w$ if and only if $v_j\leq w_j$ in $W_j$ for $1\leq j\leq m$. Therefore, it follows from Lemma~\ref{lem-lex} that $\aret_{\A_j}(e)=v_j$ and hence $\aret_\A(e)=(\aret_{\A_1}(e),\dots,\aret_{\A_m}(e))=(v_1,\dots,v_m)=v.$
\end{proof}

\begin{remark} \label{rema:multiplication}
It follows from the definition of $\aret_\A$ that we have $\aret_{v\A}(vu)=v\aret_\A(u)$ for any $u,v\in W$. Therefore, we have that $d(vu, \aret_{v\A}(vu))=d(u,\aret_\A(u))$. 
\end{remark}

\begin{theorem}\label{thm:aret-mret}
Let $W$ and $\A$ be as in Lemma~\ref{lemm:re}. 
If $\A$ is a Coxeter matroid, then $\aret_\A=\mret_\A.$ 
\end{theorem}
\begin{proof}
If $\A$ is a Coxeter matroid, then
$\A$ contains a unique minimal element in $\A$, say $v$, with respect to the ordering $\le^{u}$, i.e. $u^{-1}v\le u^{-1}w$ for all $w\in \A$. Therefore we have $\aret_{u^{-1}\A}(e)=u^{-1}v$ by Lemma~\ref{lemm:re} and hence $\aret_{\A}(u)=v$ by Remark~\ref{rema:multiplication}. Since $v=\mret_\A(u)$ by definition and $u$ is an arbitrary element of $W$, this proves $\aret_\A=\mret_\A$. 
\end{proof}

Theorem~\ref{thm:aret-mret} together with Proposition~\ref{prop_mret_minimal_distance} shows that if a subset $\A$ of a product of Weyl groups of classical Lie types is a Coxeter matroid, then $\aret_\A(u)$ gives the unique element in $\A$ closest to $u$. The following example shows that even if the subset $\A$ has the property that there is a unique element in $\A$ closest to $u$ for each $u\in W$, $\aret_\A(u)$ does not necessarily give the unique closest element unless $\A$ is a Coxeter matroid. 

\begin{example}\label{example_not_unique_closest}
Let $W=\mathfrak{S}_4$ and $\A=\{1423,2134\}$. Then $\A$ is not a Coxeter matroid because the convex hull $\Delta_\A$ of $\A$, that is the dotted red line in Figure~\ref{fig:aret_not_closest}, is not a $\Phi$-polytope, i.e. the dotted red line is not parallel to any edge of the permutohedron. On the other hand, one can check that $\A$ has the property that for each $u\in \mathfrak{S}_4$, there is a unique element in $\A$ closest to $u$. However, if we take $u=1324$ for instance, then $d(u,2134)=2$ and $d(u,1423)=3$ but $\aret_\A(u)=1423$, see Figure~\ref{fig:aret_not_closest}. Therefore, $\aret_\A(1324)$ does not give the element of $\A$ closest to $1324$. 
\begin{figure}[ht]
\begin{tikzpicture}[scale=5]
\tikzset{every node/.style={draw=blue!50,fill=blue!20, circle, thick, inner sep=1pt,font=\footnotesize}}
\tikzset{red node/.style = {fill=red!20!white, draw=red!75!white}}
\tikzset{red line/.style = {line width=0.3ex, red, nearly opaque}}

\coordinate (4231) at (1/3, 1/2, 1/6); 
\coordinate (2413) at (2/3, 1/2, 1/6); 
\coordinate (1243) at (5/6, 2/3, 1/2); 
\coordinate (2143) at (5/6, 1/2, 1/3); 
\coordinate (2134) at (5/6, 1/3, 1/2); 
\coordinate (1423) at (2/3, 5/6, 1/2); 
\coordinate (3142) at (1/3, 1/2, 5/6); 
\coordinate (1324) at (2/3, 1/2, 5/6); 
\coordinate (1234) at (5/6, 1/2, 2/3); 
\coordinate (1342) at (1/2, 2/3, 5/6); 
\coordinate (4123) at (1/2, 5/6, 1/3); 
\coordinate (4213) at (1/2, 2/3, 1/6); 
\coordinate (1432) at (1/2, 5/6, 2/3); 
\coordinate (4132) at (1/3, 5/6, 1/2); 
\coordinate (2314) at (2/3, 1/6, 1/2); 
\coordinate (3214) at (1/2, 1/6, 2/3); 
\coordinate (3124) at (1/2, 1/3, 5/6); 
\coordinate (3241) at (1/3, 1/6, 1/2); 
\coordinate (2341) at (1/2, 1/6, 1/3); 
\coordinate (2431) at (1/2, 1/3, 1/6);
\coordinate (3421) at (1/6, 1/3, 1/2); 
\coordinate (4321) at (1/6, 1/2, 1/3); 
\coordinate (3412) at (1/6, 1/2, 2/3); 
\coordinate (4312) at (1/6, 2/3, 1/2); 

\draw[thick, draw=blue!70] (1432)--(4132)--(4123)--(1423)--cycle;
\draw[thick, draw=blue!70] (4132)--(1432)--(1342)--(3142)--(3412)--(4312)--(4132);
\draw[dashed, thick, draw=blue!70] (4312)--(4321)--(4231)--(4213)--(4123);
\draw[dashed, thick, draw=blue!70] (3421)--(4321);
\draw[thick, draw=blue!70] (1342)--(1324)--(3124)--(3142);
\draw[dashed, thick, draw=blue!70] (4231)--(2431)--(2413)--(4213);
\draw[thick, draw=blue!70] (1423)--(1243)--(2143);
\draw[dashed, thick, draw=blue!70] (2143)--(2413);
\draw[thick, draw=blue!70] (1324)--(1234)--(1243);
\draw[thick, draw=blue!70] (1234)--(2134)--(2143);
\draw[thick, draw=blue!70] (2314)--(2134);
\draw[dashed, thick, draw=blue!70] (2314)--(2341)--(2431);
\draw[thick, draw=blue!70] (3412)--(3421)--(3241)--(3214)--(3124);
\draw[thick, draw=blue!70] (3214)--(2314);
\draw[dashed,thick, draw=blue!70] (3241)--(2341);

\draw[red line, dashed] (1423)--(2134);

\node[label = {[label distance = 0cm]below left:{1234}}] at (1234) {};
\node[label = {[label distance = 0cm]right:1243}] at (1243) {};
\node[label = {[label distance = 0cm]left:\textcolor{red}{1324}}] at (1324) {};
\node[label = {[label distance = 0cm]below:1342}] at (1342) {};
\node[label = {[label distance = 0cm]right:$\fbox{1423}=\aret_{\A}(\textcolor{red}{1324})$}, red node] at (1423) {};
\node[label = {[label distance = -0.2cm]above:1432}] at (1432) {};
\node[label = {[label distance = 0cm]right:$\fbox{2134}$}, red node] at (2134) {};
\node [label = {[label distance = 0cm]right:2143}] at (2143) {};
\node[label = {[label distance = -0.2cm]below:2314}] at (2314) {};
\node[label = {[label distance = 0cm]right:2341}] at (2341) {};
\node[label = {[label distance = 0cm]left:2413}] at (2413) {};
\node [label = {[label distance = -0.1cm]above:2431}] at (2431) {};
\node [label = {[label distance = 0cm]above:3124}] at (3124) {};
\node[label = {[label distance = 0cm]right:3142}] at (3142) {};
\node[label = {[label distance = -0.2cm]below:3214}] at (3214) {};
\node[label = {[label distance = 0cm]left:3241}] at (3241) {};
\node [label = {[label distance = 0cm]left:3412}] at (3412) {};
\node[label = {[label distance = 0cm]left:3421}] at (3421) {};
\node[label = {[label distance = -0.2cm]above:4123}] at (4123) {};
\node[label = {[label distance = -0.2cm]above:4132}] at (4132) {};
\node [label = {[label distance = 0cm]below:4213}] at (4213) {};
\node[label = {[label distance = -0.2cm]above:4231}] at (4231) {};
\node [label = {[label distance = 0cm]left:4312}] at (4312) {};
\node[label = {[label distance = -0.1cm]below right:4321}] at (4321) {};
\end{tikzpicture}
\caption{$\aret_{\A}(1324)$ is not the element of $\A$ closest to $1324$.}
\label{fig:aret_not_closest}
\end{figure}
\end{example}

\section{Characterization of Coxeter matroids using $d$ and $\aret_{\A}$} \label{sec:characterization}

If a subset $\A$ of $W$ is a Coxeter matroid, then for each $u\in W$, there is a unique element $\y\in \A$ such that $d(u,\y)=d(u,\A)$ by Proposition~\ref{prop_mret_minimal_distance} (1). However, this necessary condition is not sufficient for a subset $\A$ to be a Coxeter matroid as is seen in Example~\ref{example_not_unique_closest}. On the other hand, Theorem~\ref{thm:aret-mret} together with Proposition~\ref{prop_mret_minimal_distance} (2) says that if $\A$ is a Coxeter matroid of a product of Weyl groups of classical Lie types, then the unique element $\y\in\A$ closest to $u$ must be given by $\aret_{\A}(u)$. The following proposition shows that these two necessary conditions are sufficient for a subset $\A$ of $\mathfrak{S}_n$ to be a Coxeter matroid when $\A$ consists of two elements. 

\begin{proposition}\label{prop:two-elements}
Let $\A$ be a subset of $\mathfrak{S}_n$ consisting of two elements. Then $\A$ is a Coxeter matroid if and only if 
\begin{enumerate}
\item for each $u\in \mathfrak{S}_n$, there is a unique $\y\in \A$ such that $d(u,\y)=d(u,\A)$, and
\item $\y=\aret_{\A}(u)$.
\end{enumerate}
\end{proposition} 

\begin{proof}
The `only if’ part follows from Proposition~\ref{prop_mret_minimal_distance} and Theorem~\ref{thm:aret-mret} as explained above, so we will prove the `if' part. 

By Theorem~\ref{thm_GS}, it suffices to prove that the convex hull $\Delta_{\A}$ of $\A$ is a $\Phi$-polytope when the two-element subset $\A$ of $\mathfrak{S}_n$ satisfies (1) and (2) in the proposition, i.e. 
\begin{equation}\label{equation_sstar}
\begin{split}
&\text{
for each $u\in W$, $\aret_\A(u)$ gives a unique element in $\A$ such that} \\
&\text{$d(u,\aret_\A(u))=d(u,\A)$. }
\end{split} \tag{$\ast$}
\end{equation}
We note that $\Delta_\A$ is a $\Phi$-polytope if and only if $\Delta_{v\A}$ is a $\Phi$-polytope for any $v\in \mathfrak{S}_n$. Therefore, by Remark~\ref{rema:multiplication}, we may assume that one of the two elements in $\A$ is the identity element $e$ and we denote the other element by $x$, so that $\A=\{e,x\}$. Finally, we note that $\Delta_{\{e,x\}}$ is a $\Phi$-polytope if and only if $x$ is a transposition. In the following, we prove that $x$ is a transposition when $(\ast)$ is satisfied. 

We express $x$ by one-line notation: 
\[
x=x(1)x(2) \cdots x(n).
\]
We assume that $x(1)\not=1$ for simplicity. (One can see that the following argument will work for $x(k)$ instead of $x(1)$ such that $x(\ell)=\ell$ for all $\ell<k$ but $x(k)\not=k$.) Let $x(i)=1$. Since $x(1)\not=1$, we have $i>1$. We consider two elements in $\mathfrak{S}_n$: 
\begin{equation} \label{eq:uv}
\arraycolsep=1.4pt
\begin{array}{rlccccccc} 
y&=x(i)&x(1)&x(2)&\cdots &x(i-1)&x(i+1)&\cdots &x(n),\\
z&=x(1)&x(i)&x(2)&\cdots &x(i-1)&x(i+1)&\cdots &x(n).
\end{array} 
\end{equation}
We note that since $\A=\{e,x\}$ and $x(1)\not=1$, we have $\aret_\A(u)=e$ if $u(1)=1$ and $\aret_\A(u)=x$ if $u(1)=x(1)$ for $u\in \mathfrak{S}_n$. 
Therefore 
\begin{enumerate}
\item $\aret_\A(y)=e$ since $y(1)=x(i)=1$, and 
\item $\aret_\A(z)=x$ since $z(1)=x(1)$. 
\end{enumerate}
Hence by assumption~\eqref{equation_sstar}, we have that 
\begin{equation}\label{ineq1}
d(y,e) < d(y,x)\text{ and }d(z,x) < d(z,e).
\end{equation}
We note that 
\begin{enumerate}
\item[(3)] $d(y,x)=i-1$ since $y=xs_{i-1}s_{i-2}\cdots s_1$, where $s_r$ $(1\le r\le n-1)$ denotes the simple reflection switching $r$ and $r+1$, 
\item[(4)] $d(z,e)=d(y,e)+1$ since $z=ys_1$ and $z(1)=x(1)>1=x(i)=z(2)$,
\item[(5)] $d(z,x)=d(y,x)-1=i-2$ since $x(i)=1$.
\end{enumerate}
Then it follows from (3), (4), (5) above and \eqref{ineq1} that 
\begin{equation}\label{dist2}
d(y,e)=i-2,\,d(z,e)=i-1,\,\text{and }d(z,x)=i-2.
\end{equation}

Suppose that there exists some $j$ such that $3\le j<n$ and $z(j)>z(j+1)$. Then for $w=zs_j$, we have 
\[
d(w,e)=d(z,e)-1\quad\text{and}\quad d(w,x)=d(z,x)+1.
\]
It follows from these identities and \eqref{dist2} that
\begin{equation} \label{eq:dw}
d(w,e)+1=d(w,x).
\end{equation}
However, $\aret_\A(w)=x$ since $w(1)=z(1)=x(1)$. Hence $d(w,x)<d(w,e)$ by assumption~\eqref{equation_sstar}. This contradicts~\eqref{eq:dw}, so there exists no $j$ such that $3\le j<n$ and $z(j)>z(j+1)$. It follows from \eqref{eq:uv} that 
\begin{equation} \label{eq:xineq}
x(2)<x(3)<\dots <x(i-1)<x(i+1)<\dots<x(n).
\end{equation}
Since $x(i)=1$, it follows from \eqref{eq:uv}, \eqref{dist2} and \eqref{eq:xineq} that $x(1)=i$ and $x(j)=j$ for $j\not=1,i$. Thus $x$ is the transposition which transposes $1$ and $i$. 
\end{proof}

\begin{example}
Let $\A=\{1234,4231\}\subset \mathfrak{S}_4$. Then for every element $u\in\mathfrak{S}_4$, $\aret_\A(u)$ gives the unique closest element in $\A$. Figure~\ref{fig:two-elt} shows that $\aret_\A$ maps the elements colored by green to $4231$ while the elements colored by blue to $1234$. In fact, $\A$ is a Coxeter matroid since $\Delta_\A$ is a $\Phi$-polytope. 
\end{example}

\begin{figure}[ht]
\begin{tikzpicture}[scale=5]
\tikzset{every node/.style={draw=blue!50,fill=blue!20, circle, thick, inner sep=1pt,font=\footnotesize}}
\tikzset{red node/.style = {fill=red!20!white, draw=red!75!white}}
\tikzset{red line/.style = {line width=0.3ex, red, nearly opaque}}
\coordinate (4231) at (1/3, 1/2, 1/6); 
\coordinate (2413) at (2/3, 1/2, 1/6); 
\coordinate (1243) at (5/6, 2/3, 1/2); 
\coordinate (2143) at (5/6, 1/2, 1/3); 
\coordinate (2134) at (5/6, 1/3, 1/2); 
\coordinate (1423) at (2/3, 5/6, 1/2); 
\coordinate (3142) at (1/3, 1/2, 5/6); 
\coordinate (1324) at (2/3, 1/2, 5/6); 
\coordinate (1234) at (5/6, 1/2, 2/3); 
\coordinate (1342) at (1/2, 2/3, 5/6); 
\coordinate (4123) at (1/2, 5/6, 1/3); 
\coordinate (4213) at (1/2, 2/3, 1/6); 
\coordinate (1432) at (1/2, 5/6, 2/3); 
\coordinate (4132) at (1/3, 5/6, 1/2); 
\coordinate (2314) at (2/3, 1/6, 1/2); 
\coordinate (3214) at (1/2, 1/6, 2/3); 
\coordinate (3124) at (1/2, 1/3, 5/6); 
\coordinate (3241) at (1/3, 1/6, 1/2); 
\coordinate (2341) at (1/2, 1/6, 1/3); 
\coordinate (2431) at (1/2, 1/3, 1/6);
\coordinate (3421) at (1/6, 1/3, 1/2); 
\coordinate (4321) at (1/6, 1/2, 1/3); 
\coordinate (3412) at (1/6, 1/2, 2/3); 
\coordinate (4312) at (1/6, 2/3, 1/2); 
\draw[thick, draw=blue!70] (1432)--(4132)--(4123)--(1423)--cycle;
\draw[thick, draw=blue!70] (4132)--(1432)--(1342)--(3142)--(3412)--(4312)--(4132);
\draw[dashed, thick, draw=blue!70] (4312)--(4321)--(4231)--(4213)--(4123);
\draw[dashed, thick, draw=blue!70] (3421)--(4321);
\draw[thick, draw=blue!70] (1342)--(1324)--(3124)--(3142);
\draw[dashed, thick, draw=blue!70] (4231)--(2431)--(2413)--(4213);
\draw[thick, draw=blue!70] (1423)--(1243)--(2143);
\draw[dashed, thick, draw=blue!70] (2143)--(2413);
\draw[thick, draw=blue!70] (1324)--(1234)--(1243);
\draw[thick, draw=blue!70] (1234)--(2134)--(2143);
\draw[thick, draw=blue!70] (2314)--(2134);
\draw[dashed, thick, draw=blue!70] (2314)--(2341)--(2431);
\draw[thick, draw=blue!70] (3412)--(3421)--(3241)--(3214)--(3124);
\draw[thick, draw=blue!70] (3214)--(2314);
\draw[dashed,thick, draw=blue!70] (3241)--(2341);
\node[label = {[label distance = 0cm]below left:1234}, fill=Maroon!20,draw=Maroon, inner sep=1.3pt] at (1234) {};
\node[label = {[label distance = 0cm]right:1243}] at (1243) {};
\node[label = {[label distance = 0cm]left:1324}] at (1324) {};
\node[label = {[label distance = 0cm]below:1342}] at (1342) {};
\node[label = {[label distance = 0cm]right:1423}] at (1423) {};
\node[label = {[label distance = -0.2cm]above:1432}] at (1432) {};
\node[label = {[label distance = 0cm]right:2134}] at (2134) {};
\node [label = {[label distance = 0cm]right:2143}] at (2143) {};
\node[label = {[label distance = -0.2cm]below:2314}] at (2314) {};
\node[label = {[label distance = 0cm]right:2341},fill=OliveGreen!80,draw=OliveGreen] at (2341) {};
\node[label = {[label distance = 0cm]left:2413},fill=OliveGreen!80,draw=OliveGreen] at (2413) {};
\node [label = {[label distance = -0.1cm]above:2431},fill=OliveGreen!80,draw=OliveGreen] at (2431) {};
\node [label = {[label distance = 0cm]above:3124}] at (3124) {};
\node[label = {[label distance = 0cm]right:3142}] at (3142) {};
\node[label = {[label distance = -0.2cm]below:3214}] at (3214) {};
\node[label = {[label distance = 0cm]left:3241},fill=OliveGreen!80,draw=OliveGreen] at (3241) {};
\node [label = {[label distance = 0cm]left:3412},fill=OliveGreen!80,draw=OliveGreen] at (3412) {};
\node[label = {[label distance = 0cm]left:3421},fill=OliveGreen!80,draw=OliveGreen] at (3421) {};
\node[label = {[label distance = -0.2cm]above:4123},fill=OliveGreen!80,draw=OliveGreen] at (4123) {};
\node[label = {[label distance = -0.2cm]above:4132},fill=OliveGreen!80,draw=OliveGreen] at (4132) {};
\node [label = {[label distance = 0cm]below:4213},fill=OliveGreen!80,draw=OliveGreen] at (4213) {};
\node[label = {[label distance = -0.2cm]above:4231}, fill=Brown!80,draw=Brown, inner sep = 1.3pt] at (4231) {};
\node [label = {[label distance = 0cm]left:4312},fill=OliveGreen!80,draw=OliveGreen] at (4312) {};
\node[label = {[label distance = -0.1cm]below right:4321},fill=OliveGreen!80,draw=OliveGreen] at (4321) {};
\end{tikzpicture}
\caption{$\aret_\A$ for the two-element subset $\A=\{1234,4231\}$.}\label{fig:two-elt}
\end{figure}

It is natural to ask whether the assumption on the cardinality of $\A$ can be removed in Proposition~\ref{prop:two-elements}, see the question in Introduction. It is easy to check that the proposition holds for any subset $\A$ of $\mathfrak{S}_n$ when $n\le 3$ and 
a computer check shows that this is also the case when $n=4$, but we do not know whether the proposition holds for any subset $\A$ of $\mathfrak{S}_n$ when $n\ge 5$.

\end{document}